\documentclass[10pt]{amsart}
\usepackage{amsmath, amsfonts, amscd, latexsym, amsthm, amssymb}
\usepackage{hyperref}
\usepackage[margin=1.5in]{geometry}

\def \c{\mathbb{C}}
\def \z{\mathbb{Z}}
\def \r{\mathbb{R}}

\def \p{\mathbb{P}}

\def \I{\mathcal{I}}

\def \M{\mathcal{M}}
\def \A{\mathbb{A}}
\def \AA{\mathcal{A}}
\def \SS{\mathcal{S}}
\def \PP{\mathcal{P}}

\def \m{\mathfrak{m}}

\def \k{{\bf k}}

\def \K{{\bf K}}
\def \U{\mathcal{U}}

\def \G{{\bf G}}

\def \.{\cdot}

\def \covol{\textup{covol}}

\def \ratmap{\dashrightarrow}

\theoremstyle{plain}
\newtheorem{Th}{Theorem}[section]
\newtheorem{Lem}[Th]{Lemma}
\newtheorem{Prop}[Th]{Proposition}
\newtheorem{Cor}[Th]{Corollary}

\newtheorem{THM}{Theorem}
\newtheorem{COR}{Corollary}

\theoremstyle{definition}

\newtheorem{Def}[Th]{Definition}
\newtheorem{Rem}[Th]{Remark}

\begin{document}
\title{On mixed multiplicities of ideals}
\author{Kiumars Kaveh}
\address{Department of Mathematics, School of Arts and Sciences, University of Pittsburgh, 
301 Thackeray Hall, Pittsburgh, PA  15260, U.S.A.}
\email{kaveh@pitt.edu} 

\thanks{The first author is partially supported by a
Simons Foundation Collaboration Grants for Mathematicians (Grant ID: 210099) and a National Science Foundation
(Grant ID: 1200581).}

\thanks{The second author is partially supported by the Canadian Grant N 156833-12.}

\author{A. G. Khovanskii}
\address{Department of Mathematics, University of Toronto, Toronto, Canada; 
Moscow Independent University; Institute for Systems Analysis, Russian Academy of Sciences}
\email{askold@math.utoronto.ca}



\date{\today}

\maketitle
\begin{abstract}
Let $R$ be the local ring of a point on a variety $X$ over an algebraically closed field $\k$. We make a connection between the notion of mixed 
(Samuel) multiplicity of $\m$-primary ideals in $R$ and intersection theory of subspaces of rational functions on $X$ which deals with the number of solutions of systems of equations. From this we readily deduce several properties of mixed multiplicities. In particular, we prove a (reverse) Alexandrov-Fenchel inequality for mixed multiplicities due to Teissier and Rees-Sharp. 
As an application in convex geometry one obtains a proof of a (reverse) Alexandrov-Fenchel inequality for covolumes of convex bodies inscribed in a convex cone. 
\end{abstract}

\tableofcontents

\section*{Introduction}
Consider the local ring $R = \mathcal{O}_{X, a}$ of a point $a$ on an $n$-dimensional algebraic variety $X$ over an algebraically 
closed field $\k$. The set of elements of $R$ vanishing at $a$ is the unique maximal ideal $\m$. An ideal $I \subset R$ is called $\m$-primary 
if it contains a power of the maximal ideal $\m$, or equivalently if the subvariety it defines around $a$ consists of $a$ itself.
In this paper we make a connection between the notion of (mixed) multiplicity of $\m$-primary ideals in $R$ and the intersection theory of subspaces of 
rational functions on $X$ (dealing with number of solutions of systems of algebraic equations), as developed in \cite{KKh-MMJ, KKh-CMB}. This intersection theory is a birational (and simpler) version of intersection theory of Cartier divisors and linear systems. Each finite dimensional subspace of rational functions is analogue of a Cartier divisor and intersection index of $n$ subspaces is analogue of the intersection number of $n$ divisors. In fact the intersection theory of subspaces of rational functions can be identified with the intersection theory of Cartier b-divisors of Shokurov (see \cite{KKh-CMB}).

We will see that several basic properties of mixed multiplicities of ideals immediately follow from similar/related properties of the intersection index of subspaces of rational functions. In particular we prove the following (reverse) Alexandrov-Fenchel inequality. 

\begin{THM} \label{th-intro-AF}
Let $I_1, I_2, \ldots, I_n$ be $\m$-primary ideals in the local ring $R$. Then:
$$e(I_1, I_1, I_3, \ldots, I_n) e(I_2, I_2, I_3, \ldots, I_n) \geq e(I_1, I_2, I_3, \ldots, I_n)^2,$$
where $e(I_1, \ldots, I_n)$ denotes the mixed (Samuel) multiplicity of the $I_i$.
\end{THM}
This inequality has been proved in \cite{Teissier1} and \cite{RS} for general Noetherian local rings.

To simplify the presentation assume that X is an irreducible variety. Consider the collection $\K(X)$ of all the 
nonzero finite dimensional $\k$-subspaces of rational functions on $X$. The set $\K(X)$ is equipped with a natural product: for two subspaces $L, M \in \K(X)$, the product $LM$ is the subspace spanned by all the $fg$ where 
$f \in L$ and $g \in M$. With this product $\K(X)$ is a commutative semigroup (without cancellation property). Let $L_1, \ldots, L_n$ be subspaces in $\K(X)$, in \cite{KKh-MMJ} we associate a nonnegative integer $[L_1, \ldots, L_n]$ to 
the subspaces $L_i$ and call it their intersection index. It is defined to be the number of solutions $x$ of a system 
$f_1(x) = \cdots = f_n(x) = 0$ where $f_i \in L_i$ are general elements and $x$ lies in a certain nonempty Zariski open subset $U$ of $X$ (depending on the $L_i$). In \cite{KKh-MMJ, KKh-CMB} it is shown that, the intersection index is well-defined and moreover it is multi-additive with respect to the product of subspaces. It follows that the intersection index extends to a multi-additive integer valued function on the Grothendieck group $\G(X)$ of the semigroup $\K(X)$ (Definition \ref{def-G(X)} and paragraph before it). 

We regard $\G(X)$, together with its intersection index, as an extension of the intersection theory of Cartier divisors on complete varieties.
As mentioned above, the Grothendieck group of the semigroup of subspaces of rational functions 
is naturally isomorphic to the group of Shokurov's Cartier b-divisors.
Naturally, the well-definedness and multi-additivity of the intersection index can be deduced from the usual intersection theory on a product of projective spaces (\cite{KKh-CMB}), when $\k=\c$ it can be proved topologically as well (\cite{KKh-MMJ}).

Let $R$ be a general Noetherian local ring of dimension $n$ with maximal ideal $\m$. In commutative algebra,
one shows that for an $\m$-primary ideal $I$, the Hilbert-Samuel function $H_I(k) = \ell(R / I^k)$ is a polynomial of degree $n$, for sufficiently large values of $k$, where $\ell$ denotes the length of $R$-modules. One then defines the Samuel multiplicity
$e(I)$ of $I$ to be the limit:
$$n!~\lim_{k \to \infty} \frac{H_I(k)}{k^n},$$
that is, $n!$ times the leading coefficient of the Hilbert-Samuel polynomial.
Moreover, if $I_1, \ldots, I_n$ are $\m$-primary ideals, one shows that the function 
$e(I_1^{k_1} \cdots I_n^{k_n})$ is a polynomial in the $k_i$ of degree $n$. Then the Samuel mixed multiplicity $e(I_1, \ldots, I_n)$ is usually defined 
as the coefficient of $k_1 \cdots k_n$ in this polynomial divided by $n!$.

In the present paper, for the local ring $R = \mathcal{O}_{X, a}$ of a point $a$ on a variety $X$, 
we take a more geometric approach to the definition of mixed multiplicity of ideals. We start with the classical notion of the 
multiplicity $e_a(f_1, \ldots, f_n)$ of a system of $n$ algebraic equations $f_1(x) = \cdots = f_n(x) = 0$ which has $a$ as an isolated root (Section \ref{sec-multi-root}). Intuitively $e_a(f_1, \ldots, f_n)$ is the number of roots that are born around $a$ if one 'perturbs' the functions $f_i$. 
Then given $\m$-primary ideals $I_1, \ldots, I_n$ we see that for a 'generic' $n$-tuple 
$(f_1, \ldots, f_n) \in I_1 \times \cdots \times I_n$ the multiplicity $e_a(f_1, \ldots, f_n)$ is the same, and in fact is the minimum multiplicity among all the systems in $I_1 \times \cdots \times I_n$ with isolated root at $a$ (Theorem \ref{th-multi-ideals}).
We define the mixed multiplicity $e(I_1, \ldots, I_n)$ to be the multiplicity of a generic system from 
$I_1 \times \cdots \times I_n$. This is known in commutative algebra in the following form: Let $I_1, \ldots, I_n$ be $\m$-primary ideals in a local domain $R$ and let $x_i \in I_i$ be generic elements, then the Samuel mixed multiplicity $e(I_1, \ldots, I_n)$ is equal to the Samuel multiplicity $e(I)$ where $I = (x_1, \ldots, x_n)$ (see \cite{Teissier2} and also\cite{Rees}).

The above theorem (Theorem \ref{th-multi-ideals}) and definition of mixed multiplicity 
are similar to Theorem \ref{th-int-index-well-def} and Definition \ref{def-int-index}
of the intersection index of subspaces of rational functions (Section \ref{sec-int-index}). The relationship between the multiplicity of a system of $n$ equations at $a$ and the mixed multiplicity of an $n$-tuple $(I_1, \ldots, I_n)$ of $\m$-primary ideals is analogous to the relationship between the number of solutions of a system of $n$ equations on a variety and the intersection index of an $n$-tuple of finite dimensional subspaces of rational functions $(L_1, \ldots, L_n)$.

Without loss of generality we can assume $X$ is an affine variety embedded in an affine space $\A^N$ and the point 
$a$ is the origin $o$. Let $I$ be an $\m$-primary ideal in $R = \mathcal{O}_{X, o}$. 
Let $m>0$ be an integer. To $m$ we associate the subspaces $A^{(m)}$ and $I^{(m)}$ in $R$ respectively consisting of functions which are restrictions of polynomials of degree at most $m$ on $\A^N$, and functions in $I$ which are restrictions of polynomials of degree at most $m$ on $\A^N$. Our main theorem relating the mixed multiplicity of ideals and intersection index is the following (Theorem \ref{Groth-homomorphism}, Theorem \ref{th-multi-int-index} and Corollary \ref{cor-multi-int-index}).

\begin{THM} \label{th-intro-multi-int-index}
Let $I_1, \ldots, I_n$ be $\m$-primary ideals in the local ring $R$. Then for any sufficiently large $m > 0$, and  
any generic system $(f_1, \ldots, f_n) \in I^{(m)}_1 \times \cdots \times I^{(m)}_n$, we have 
$e_o(f_1, \ldots, f_n) = [A^{(m)}, \ldots, A^{(m)}] - [I_1^{(m)}, \ldots, I_n^{(m)}]$. In other words,
$$e(I_1, \ldots, I_n) = [A^{(m)}, \ldots, A^{(m)}] - [I_1^{(m)}, \ldots, I_n^{(m)}].$$
Moreover,
$$e(I_1, \ldots, I_n) = -[I_1^{(m)} / A^{(m)}, \ldots, I_n^{(m)} / A^{(m)}],$$
where the $I_i^{(m)} / A^{(m)}$ are elements of the Grothendieck group $\G(X)$ (analogue of difference of two divisors). {In fact, we show that the map $I \mapsto I^{(m)} / A^{(m)}$ is a homomorphism from the semigroup of 
$\m$-primary ideals in $R$ into the Grothendieck group $\G(X)$ of $\K(X)$ which by above, up to a minus sign, sends the mixed multiplicity to intersection index.} 
\end{THM}

From this we readily get proofs of the following (Theorem \ref{th-mixed-multi-additive}, Corollary \ref{cor-multi-polynomial} and Theorem \ref{th-Samuel-vs-ours}). 
Also using a version of Hodge inequality for intersection index (Theorem \ref{th-Hodge}) we get a proof of Theorem \ref{th-intro-AF}.
\begin{COR} \label{cor-intro-multi-properties}
\begin{itemize}
\item[(1)] Mixed multiplicity is multi-additive.
That is, if $I_1', I_1'', I_2, \ldots, I_n$ are $\m$-primary ideals then:
$$e(I_1'I_1'', I_2, \ldots, I_n) = e(I_1', I_2, \ldots, I_n) + e(I_1'', I_2, \ldots, I_n).$$
\item[(2)] Multiplicity is a homogeneous polynomial of degree $n$. That is,
if $I_1, \ldots, I_n$ are $\m$-primary ideals in $R$ then the multiplicity function 
$$e(I_1^{k_1} \cdots I_n^{k_n})$$ 
is a homogeneous polynomial of degree $n$ in $k_1, \ldots, k_n$.
\item[(3)] The notion of mixed multiplicity as the minimum multiplicity at $a$ for all the systems of equations from $I_1 \times \cdots \times I_n$ (with isolated root at $a$)
coincides with the notion of multiplicity in the sense of Samuel.
\end{itemize}
\end{COR}

To make the relationship/analogy between $\m$-primary ideals and finite dimensional subspaces of rational functions more transparent 
we consider the semigroup of all $\m$-primary ideals in $R$ with product of ideals. Let us identify 
two primary ideals $I, J$ if the restrictions of $I$ and $J$ to each irreducible component $X_i$ coincide. Let 
$\K_{ideal}(R)$ be the semigroup of $\m$-primary ideals modulo this identification of ideals. We denote the Grothendieck group of the semigroup $\K_{ideal}(R)$ by $\G_{ideal}(R)$. The map $I \mapsto I^{(m)} / A^{(m)}$ above (Theorem \ref{th-intro-multi-int-index}) gives a natural homomorphism from $\G_{ideal}(R)$ into $\G(X)$ (Theorem \ref{Groth-homomorphism}). We summarize the relationship/analogy between $\m$-primary ideals and their mixed multiplicities, and finite dimensional subspaces of rational functions and their intersection indices in the following table:\\

\begin{tabular}{|c|c|} 
\hline
$\m$-primary ideals in $R = \mathcal{O}_{X, a}$ & finite dimensional subspaces of $\k(X)$\\
\hline 
$\K_{ideal}(R)$ & $\K(X)$\\
\hline
$\G_{ideal}(R)$ & $\G(X)$\\
\hline
mixed multiplicity & intersection index\\ 
\hline
reverse Alexandrov-Fenchel & Alexandrov-Fenchel\\
inequality & inequality\\
\hline
Bernstein-Kushnirenk theorem & Bernstein-Kushnirenko theorem\\
for mixed multiplicities of monomial ideals & for intersection indices of monomial subspaces\\
\hline
\end{tabular}
\vspace{.5cm} 

Finally we apply the above results to obtain results about covolumes of convex bodies inscribed in a cone. In particular we prove a (reverse) Alexandrov-Fenchel inequality for mixed covolumes which we briefly explain below. We should point out that the connection between covolumes and (mixed) multiplicities of monomial ideals is not new and already well-known to 
B. Teissier (see \cite[Appendix]{Teissier3} which also mentions Alexandrov-Fenchel inequality for covolumes in paragraph before Theorem 8.10).

Let $C$ be a closed strictly convex cone of dimension $n$ with apex at the origin in the Euclidean space $\r^n$ 
(strictly convex cone means it is convex and does not contain any line through the origin).
We call a subset $\Gamma \subset C$ a $C$-convex region, if $\Gamma$ is closed and 
convex and moreover for any $x \in \Gamma$ and $y \in C$ we have $x+y \in \Gamma$. 
We call a $C$-convex region cobounded if the completment $C \setminus \Gamma$ is bounded. 
We call the volume of $C \setminus \Gamma$ the {\it covolume} of the convex region $\Gamma$ and 
denote it by $\covol(\Gamma)$. Corollary \ref{cor-intro-multi-properties} then implies 
(Theorem \ref{th-covol-poly}):

\begin{COR} \label{cor-intro-covol-poly}
The covolume of cobounded convex regions is a homogeneous polynomial of degree $n$. More precisely, 
let $\Gamma_1, \ldots, \Gamma_n$ be $C$-convex regions. Then the function:
$$P(\lambda_1, \ldots, \lambda_n) = \covol(\lambda_1 \Gamma_1 + \cdots + \lambda_n\Gamma_n),$$ is 
a homogeneous polynomial of degree $n$ in the $\lambda_i \geq 0$.
\end{COR}

Imitating the definition of mixed volume of convex bodies, one can then define mixed covolume of convex regions: 
let $\mathcal{C}$ denote the collection of all cobounded $C$-convex regions. Then there exists a unique function $V: \mathcal{C}^n \to \r$ such that: (1) $V$ is linear in each argument, (2) for any cobounded 
$C$-convex region $\Gamma$ we have $V(\Gamma, \ldots, \Gamma) = \covol(\Gamma)$. For $C$-convex regions 
$\Gamma_1, \ldots, \Gamma_n$ we call $V(\Gamma_1, \ldots, \Gamma_n)$ the mixed covolume of 
the $\Gamma_i$. 

The local ring $R$ of the unique fixed point of an $n$-dimensional affine toric variety is (the localization of) a monomial algebra and to it 
one associates an $n$-dimensional rational cone $C_R \in \r^n$. With an $\m$-primary monomial ideal $I \subset R$ one can associate the convex hull 
$\Gamma_I$ of all the exponents of monomials in $I$. The set $\Gamma_I$ is a $C_R$-convex region (Section \ref{sec-convex-geo}). 
According to the local 
Bernstein-Kushnirenko theorem (Theorem \ref{th-BK-local}) one has the following formula for the mixed multiplicity of $\m$-primary monomial ideals: Let $I_1, \ldots, I_n$ be 
$\m$-primary ideals in $R$ then $$e(I_1, \ldots, I_n) = n!V(\Gamma_{I_1}, \ldots, \Gamma_{I_n}).$$ 
From Theorem \ref{th-intro-AF} and the local Bernstein-Kushnirenko theorem we obtain the following (Theorem \ref{th-AF-covolume}).

\begin{COR}[Alexandrov-Fenchel inequality for covolumes] \label{cor-intro-AF-convex}
Let $\Gamma_1, \ldots, \Gamma_n$ be $C$-convex regions. Then the following inequality holds between the 
mixed covolumes:
$$V(\Gamma_1, \Gamma_1, \Gamma_3, \ldots, \Gamma_n) 
V(\Gamma_2, \Gamma_2, \Gamma_3, \ldots, \Gamma_n) \geq V(\Gamma_1, \Gamma_2, \ldots, \Gamma_n)^2.$$
\end{COR}
Corollary \ref{cor-intro-AF-convex} can also be proved in a purely convex geometric way using the classical Alexandrov-Fenchel inequality for mixed volumes of convex bodies (see \cite{Askold-Vladlen}, also for background material on the classical Alexandrov-Fenchel inequality see \cite{BZ}).

To make the paper accessible to a wider range of audience throughout the paper we have tried to recall most of the background material. 

\section{Intersection theory of subspaces of rational functions} \label{sec-int-index}
In \cite{KKh-MMJ, KKh-CMB} the authors develop an intersection theory for finite dimensional subspaces of rational functions.
In this section we recall basic facts from this theory. It is closely related to the intersection theory of Cartier divisors and linear systems (and more precisely Cartier b-divisors of Shokurov). 


We should point out that there is a slight logical imprecision in our presentation of the main results concerning the intersection index in the paper \cite{KKh-MMJ}. There the intersection index is defined only for an irreducible variety, but in the restriction theorem 
(see Theorem \ref{th-int-index-rest} below) we talk about the intersection index on a subvariety $X_{\bf f}$ which could possibly be reducible. Nevertheless this can be easily resolved. One only has to slightly adjust definitions and statements to hold for reducible varieties as well. The same arguments can be used or repeated almost word by word (as an example see the proof of Theorem \ref{th-int-index-well-def} below in which we have included the boring details).

In what follows $X$ is a possibly reducible algebraic variety over an algebraically closed field $\k$, such that all its irreducible components $X_i$  have the same dimension $n$. Let $\k(X)$ denote the algebra of rational functions on $X$. {One shows that $f \mapsto (f_{|X_1}, \ldots, f_{|X_r})$ gives a $\k$-algebra 
isomorphism between $\k(X)$ and $\bigoplus_{i=1}^r \k(X_i)$.}

The collection of all finite dimensional subspaces of rational functions $\k(X)$ has a product. 
For two finite dimensional subspaces $L, M$ we define 
$LM$ to be the subspace spanned (over $\k$) by all the products $fg$ where $f \in L$ and $g \in M$. 

\begin{Def} \label{def-K(X)} 
{If $X$ is irreducible we let $\K(X)$ to be the collection of all nonzero finite dimensional subspaces of the field of 
rational functions $\k(X)$. When $X$ is not irreducible we define $\K(X)$ to be the collection of all finite dimensional subspaces of the algebra of rational functions $\k(X)$ whose restriction to each irreducible component $X_i$ is nonzero, modulo the following relation: we identify two subspaces $L, M$ if $L_{|X_i} = M_{|X_i}$ for every $i$.
In other words, $\K(X)$ is the direct sum $\bigoplus_{i=1}^r \K(X_i)$. Clearly $\K(X)$ is a commutative semigroup with respect to product of subspaces.}
\end{Def}

\begin{Rem} \label{rem-L-vs-sum-Li}
{Suppose that $L$ is a finite dimensional subspace of the algebra $\k(X)$ such that 
$L_{|X_i}$ is nonzero for any $i$. Then by definition of $\K(X)$, $L$ and $\bigoplus_{i=1}^r L_{|X_i}$ represent the same element in the semigroup $\K(X)$ (it is not difficult to find an example of a subspace $L \subset \k(X)$ such that $L \neq \bigoplus_{i=1}^r L_{|X_i}$).
In the rest of the paper by abuse of notation we may write $L \in \K(X)$ to mean the element in $\K(X)$ represented by $L$ i.e. $\bigoplus_{i=1}^r L_{|X_i}$.}
\end{Rem}

We say that a hypersurface $H \subset X$ is a {\it pole} for $L \in \K(X)$ if it is a pole of some function in $L$. 
Clearly the union of poles of $L$ is a subvariety of $X$ of smaller dimension (possibly empty). The {\it base locus} 
of $L \in \K(X)$ is the collection of all points at which all the functions from $L$ vanish. It is also a subvariety of smaller dimension 
(possibly empty). 

Let ${\bf L} = (L_1, \ldots, L_n)$ be an $n$-tuple of elements from $\K(X)$. let $\Sigma_{\bf L}$
denote the union of all the poles and base loci of the $L_i$ as well as the singular locus of $X$. It is a subvariety in $X$ of smaller dimension (possibly empty). Let $\Sigma \subset X$ be any subvariety 
containing $\Sigma_{\bf L}$ and such that $\dim(\Sigma) < n$.

\begin{Th} \label{th-int-index-well-def}
There is a nonempty Zariski open set ${\bf U} \subset L_1 \times \cdots \times L_n$ 
such that for any ${\bf f} = (f_1, \ldots , f_n) \in {\bf U}$ the following holds: The system
\begin{equation} \label{equ-system-Sigma}
f_1(x) = \cdots = f_n(x) = 0, \quad x \in X \setminus \Sigma
\end{equation}
has finitely many roots, all of them are simple and their number is independent of
the choices of $\Sigma$ and ${\bf f} \in {\bf U}$.
\end{Th}
\begin{proof}
The proof is based on the following arguments. From \cite{KKh-MMJ} and  \cite{KKh-CMB} 
we know that the statement is true if $X$ is irreducible. Now let $X = \bigcup_{i} X_i$ where the $X_i$ are irreducible components of $X$. By assumption $\dim(X_i) = n$ for each 
component $X_i$. For any $i$ consider the $n$-tuple of subspaces ${\bf L}_i = 
({L_1}_{|X_i}, \ldots, {L_n}_{|X_i})$ on $X_i$
and a subvariety $\Sigma_i \subset X_i$ such that $\dim(\Sigma_i) < n$, 
$\Sigma_{{\bf L}_i} \subset \Sigma_i$ and $\bigcup_{i \neq j} (X_i \cap X_j) \subset \Sigma_i$.
Applying the statement for irreducible $X_i$ equipped with the $n$-tuple of subspaces
${\bf L}_i$ and with the subvarieties $\Sigma_i$ we obtain the following: There is a nonempty Zariski open set $\tilde{\bf U}_i \subset L_1 \times \cdots \times L_n$ such that all roots of the system
\begin{equation} \label{equ-system-Sigma_i}
f_1(x) = \cdots = f_n(x) = 0, \quad x \in X_i \setminus \Sigma_i
\end{equation}
for ${\bf f} = (f_1,\ldots,f_n) \in \tilde{\bf U}_i$, are simple and their number $N_i$ is independent of the
choices of $\Sigma_i$ and ${\bf f} \in \tilde{\bf U}_i$. (Here we use that the preimage 
$\tilde{\bf U}_i = \pi^{-1}({\bf U}_i)$ of a Zariski
open set ${\bf U}_i \subset {L_1}_{|X_i} \times \cdots \times {L_n}_{|X_i}$ under the restriction map 
$\pi_i : L_1 \times \cdots \times L_n \to {L_1}_{|X_i} \times \cdots \times {L_n}_{|X_i}$ 
is a Zariski open set.) Now we return to the proof of the Theorem 
\ref{th-int-index-well-def}. Let $\Sigma \subset X$ be such that $\dim(\Sigma) < n$ and $\Sigma_{\bf L} \subset \Sigma$. For each component $X_i$  
put $\Sigma_i = \Sigma \cap X_i$. 
The intersection $X_i \cap X_j$ for $i \neq j$ 
belongs to the set of singular points of $X$, so $\bigcup_{i \neq j} (X_i \cap X_j) \subset \Sigma_i$. According to the previous arguments there is $\tilde{\bf U}_i$ such that for any ${\bf f} \in \tilde{\bf U}_i$ all roots of the system \eqref{equ-system-Sigma_i} are simple and their number $N_i$ is independent of ${\bf f}$ and $\Sigma$. To complete the proof it is enough to take ${\bf U} = \bigcap \tilde{\bf U}_i$. For any ${\bf f} \in {\bf U}$ all the roots of the system \eqref{equ-system-Sigma} are simple and their number is equal to $N = \sum_i N_i$.
\end{proof}

\begin{Def} \label{def-int-index}
With notation as in Theorem \ref{th-int-index-well-def} 
we call the number of solutions $\#\{x \in X \setminus \Sigma \mid f_1(x) = \cdots = f_n(x) = 0\}$, 
the {\it intersection index of $(L_1, \ldots, L_n)$} and denote it by
$[L_1, \ldots, L_n]$.
\end{Def}

Moreover one has the following:
\begin{Th} \label{th-int-index-leq}
For any $(g_1,\ldots,g_n) \in L_1 \times \cdots \times L_n$ and for $\Sigma_L$ equal to the union of all the poles and the base loci of the $L_i$ the number of isolated solutions $x \in X \setminus \Sigma_L$ 
of the system $g_1(x) = \cdots = g_n(x) = 0$ is less than or equal to $[L_1,\ldots,L_n]$.
\end{Th}

Let $L_1, \ldots, L_n \in \K(X)$ be finite dimensional subspaces of rational functions on $X$. For $0 < k < n$ 
and ${\bf f} = (f_1, \ldots, f_k) \in L_1 \times \cdots \times L_k$ define
$$X_{\bf f} = \{ x \in X \setminus \Sigma \mid f_1(x) = \cdots = f_k(x) = 0\}.$$  
Here as above $\Sigma$ is a subvariety of $X$ with $\dim(\Sigma) < n$ and contains the union of poles and base loci of the $L_i$.
 
\begin{Th}[Restriction theorem] \label{th-int-index-rest}
There exists a nonempty Zariski open subset ${\bf U}$ of $L_{k+1} \times \cdots \times L_n$ such that for any 
${\bf f} = (f_{k+1}, \ldots, f_n) \in {\bf U}$ we have: 
\begin{itemize}
\item[(1)] Each irreducible component of $X_{\bf f}$ is a variety of dimension $k$. 
\item[(2)] The intersection index $[{L_{1}}_{|X_{\bf f}}, \ldots, {L_{k}}_{|X_{\bf f}}]$ on the variety $X_{\bf f}$  coincides with 
$[L_1, \ldots, L_n]$.
\end{itemize}
\end{Th}

\begin{Th}[Multi-additivity] \label{th-int-index-multiadd}
Let $L_1', L_1'', L_2, \ldots, L_n$ be subspaces of rational functions in $\K(X)$. Then 
$$[L_1'L_1'', L_2, \ldots, L_n] = [L_1', L_2, \ldots, L_n] + [L_1'', L_2, \ldots, L_n].$$ That is, on $\K(X)$, 
the intersection index is multi-additive in each argument
\end{Th}
Theorems \ref{th-int-index-leq} to \ref{th-int-index-multiadd} were proved in \cite{KKh-MMJ} with a slight logical imprecision which can be easily resolved as discussed above.

Each $L \in \K(X)$ gives rise to a rational map $\Phi_L$ from $X$ to the projective space $\p(L^*)$, where $L^*$ is the dual vector space of $L$, as follows: Let $x \in X$ be a point not in the base locus of $L$ or a pole of $L$. Define $\Phi_L(x)$ be the point in $\p(L^*)$ represented by the linear function $f \mapsto f(x)$.
\begin{Def}[Kodaira map] \label{def-Kodaira-map}
We call the rational map $\Phi_L: X \ratmap \p(L^*)$ the {\it Kodaira map of $L$}.
\end{Def}

From the definition of the rational map $\Phi_L$ we immediately have the following.
Let $Y_L$ denote the closure of the image of $X$ under $\Phi_L$ in the projective space $\p(L^*)$.
\begin{Prop}(Intersection index and degree of a projective subvariety) \label{prop-int-index-degree}
Let us assume that $\Phi_L$ is a birational isomorphism between $X$ and $Y_L$. Then degree of $Y_L$, as a subvariety of $\p(L^*)$, is equal to $[L, \ldots, L]$.
\end{Prop}

Let $K$ be a commutative semigroup (whose operation we denote by
multiplication). $K$ is said to have the {\it cancellation property} if
for $x,y,z \in K$, the equality $xz=yz$ implies $x=y$. Any
commutative semigroup $K$ with the cancellation property can be extended
to an abelian group $G(K)$ consisting of formal quotients $x/y$, $x,
y \in K$. For $x,y,z,w \in K$ we identify the quotients $x/y$ and
$w/z$, if $xz = yw$.

Given a commutative semigroup $K$ (not necessarily with the cancellation
property), we can get a semigroup with the cancellation property by
considering the equivalence classes of a relation $\sim$ on $K$:
for $x, y \in K$ we say $x \sim y$ if there is $z \in K$ with $xz = yz$. The
collection of equivalence classes $K / \sim$ naturally has structure
of a semigroup with the cancellation property. Let us denote the group
of formal quotients of $K / \sim$ again by $G(K)$. It is called the {\it
Grothendieck group of the semigroup $K$}. The map which sends $x \in K$ to its
equivalence class $[x] \in K / \sim$ gives a natural homomorphism
$\phi: K \to G(K)$.

The Grothendieck group $G(K)$ together with the homomorphism $\phi: K \to
G(K)$ satisfies the following universal property: for any other
group $G'$ and a homomorphism $\phi': K \to G'$, there exists a unique
homomorphism $\psi: G(K) \to G'$  such that $\phi' = \psi \circ
\phi$.

\begin{Def}[Grothendieck group of subspaces of rational functions] \label{def-G(X)}
We denote by $\G(X)$ the Grothendieck group of the semigroup $\K(X)$ with respect to the product of subspaces.
By Theorem \ref{th-int-index-multiadd}, the intersection index extends to a multi-additive function on the abelian group $\G(X)$.
\end{Def}

{Consider the case where $X$ is irreducible and hence $\k(X)$ is a field. An element $f \in \k(X)$ is said to be {\it integral} over a subspace $L \in \K(X)$ if 
there exist an integer $m$ and elements $a_j \in L^j, j = 1, \ldots, m$ such that:
$$f^{m} +a_1f^{m-1} +a_2 f^{m-2} + \cdots +a_{m-1}f+ a_{m} = 0.$$
Let $\overline{L}$ denote the collection of all $f \in \k(X)$ which are integral over $L$.
It is a standard fact from commutative algebra that $\overline{L}$ is a finite dimensional subspace of $\k(X)$ containing $L$. It is called the {\it completion of $L$}. 

Now consider the general case where $X$ is not necessarily irreducible. For $L \in \K(X)$ let us define 
the {\it completion} $\overline{L}$ to be $\overline{L} = \bigoplus_{i=1}^r \overline{L_{|X_i}}$. 
Clearly, $\overline{L}$ is the collection of all $f \in \k(X)$ such that $f_{|X_i}$ is integral over 
$L_{|X_i}$ for every $i$.
One can give the following characterization of equivalence of elements in the commutative semigroup $\K(X)$ in terms of completion of subspaces (see \cite{SZ}):}
\begin{Th} \label{th-int-closure-subspace}
For $L \in \K(X)$, the completion $L$ is the largest subspace which is equivalent to $L$: That is, (1) $L \sim \overline{L}$ and (2) 
if for $M \in \K(X)$ we have $M \sim L$ then $M \subset \overline{L}$.
\end{Th}
{A subspace $L \in \K(X)$ is called {\it complete} if $L = \overline{L}$. 
If $L$ and $M$ are complete subspaces, then $LM$ is not necessarily
complete. For two complete subspaces $L, M \in \K(X)$,
define $$L * M = \overline{LM}.$$ The collection of complete
subspaces together with $*$ is a semigroup with the cancellation
property. Theorem \ref{th-int-closure-subspace} in fact shows that $L \mapsto
\overline{L}$ gives an isomorphism between the quotient semigroup
$\K / \sim$ and the semigroup of complete subspaces (with $*$).}


Finally, one has an analogue of the Hodge inequality for intersection index of subspaces on irreducible surfaces. 
Using the theory of Newton-Okounkov bodies one can reduce it to the classical isoperimetric inequality for convex bodies in 
the Euclidean plane $\r^2$ (see \cite{KKh-Annals}). This inequality easily implies the usual Hodge inequality 
for intersection numbers of curves on projective surfaces.

\begin{Th}[A version of Hodge inequality] \label{th-Hodge}
Let $X$ be an irreducible surface. Let $L, M$ be finite dimensional subspaces of rational functions on $X$. Then we have
\begin{equation} \label{equ-Hodge}
[L, L][M, M] \leq [L, M]^2.
\end{equation}
\end{Th}
In Section \ref{sec-AF-ideal} we will use this Hodge inequality to prove an Alexandrov-Fenchel inequality for mixed multiplicities of 
$\m$-primary ideals in the local ring of a point on an algebraic variety.

\section{Semigroup of $\m$-primary ideals and subspaces of rational functions} 
\label{sec-semigp-ideal-subspace}
{Analogous to the Grothendieck group of the semigroup of subspaces $\K(X)$ (Section \ref{sec-int-index}) in this section we consider the Grothendieck group of $\m$-primary ideals. 
Let $R = \mathcal{O}_{X, a}$ be the local ring of a point $a$ on a (possibly reducible) variety $X$ of pure dimension $n$ (i.e. all its irreducible components have dimension $n$). Let $\m$ be the unique maximal ideal of $R$. Recall that an ideal $I$ is $\m$-primary if it contains a power of the maximal ideal $\m$. 

Consider the collection $\K_{ideal}(R)$ of all the $\m$-primary ideals in $R$. The set $\K_{ideal}(R)$ is a semigroup with respect to product of ideals. As for any commutative semigroup, we say that $I, J \in \K_{ideal}(R)$ are equivalent, denoted $I \sim J$, if there is $M \in \K_{ideal}(R)$ with $IM = JM$. We denote the Grothendieck group of $\K_{ideal}(X)$ by $\G_{ideal}(X)$.


\begin{Rem} \label{rem-int-closed-ideal}
{Let us consider the case where $X$ is irreducible and hence $R$ is an integral domain (analogous statements to statements below may not hold if $R$ is not an integral domain).
Given an ideal $I \in \K_{ideal}(R)$ and $f \in R$ we say that $f$ is {\it integral over $I$} if there exist an integer $m$ and elements $a_i \in I^i, i = 1, \ldots, m$, such that
$$f^m +a_1f^{m-1} +a_2 f^{m-2} + \cdots +a_{m-1}f+ a_m = 0.$$
It is a standard fact from commutative algebra that the set of all elements in $R$ which are integral over $I$ is an ideal called the {\it integral closure of $I$} and denoted $\overline{I}$ (see \cite[Appendix 4]{SZ} and \cite[Chapter 1]{SH}). Clearly if $I$ is $\m$-primary then $\overline{I}$ is also $\m$-primary. 
Similar to Theorem \ref{th-int-closure-subspace} one can give the following characterization of equivalence of elements in the commutative semigroup $\K_{ideal}(R)$ in terms of integral closure of ideals:
For $I \in \K_{ideal}(R)$, the completion $I$ is the largest $\m$-primary ideal which is equivalent to $I$: That is, (1) $I \sim \overline{I}$ and (2) if for $J \in \K_{ideal}(R)$ we have $I \sim J$ then $J \subset \overline{I}$.

Let us call an $\m$-primary ideal $I$ {\it integrally closed} if $I = \overline{I}$. If $I, J$ are integrally closed $\m$-primary ideals then $IJ$ is not necessarily integrally closed. For two integrally closed $I, J \in \K_{ideal}(R)$ define: $$I * J = \overline{IJ}.$$
The collection of integrally closed $\m$-primary ideals together with $*$ is a semigroup with the cancellation property. The above shows that $I \mapsto \overline{I}$ gives an isomorphism between the quotient semigroup $\K_{ideal}(R) / \sim$ and the semigroup of integrally closed $\m$-primary ideals (with $*$).}
\end{Rem}

Lets go back to the general case where $X$ is possibly reducible.
Without loss of generality assume that $X$ is affine and embedded in an affine space $\A^N$ and $a=o$ is the origin. 
We now show that the Grothendieck group $\G_{ideal}(R)$ of $\m$-primary ideals in $R$ can be naturally mapped into the Grothendieck group $\G(X)$ of finite dimensional subspaces.

For each $m > 0$ let $\AA^{(m)}$ and $\PP^{(m)}$ 
denote the vector spaces of polynomials in $\AA = \k[x_1, \ldots, x_N]$ of degree less than or equal to $m$, and the vector space of 
homogeneous polynomials of degree $m$ respectively.  
Also let $A^{(m)}$ (respectively $P^{(m)}$) denote the set of all $f \in R$ which are restriction of a polynomial of 
degree at most $m$ (respectively degree equal to $m$). 

\begin{Def} \label{def-I^(m)}
{For an $\m$-primary ideal $I \subset R$ put $I^{(m)} = A^{(m)} \cap I$. The subspace $I^{(m)}$ represents an element of $\K(X)$ which we also denote by $I^{(m)}$.}
\end{Def}

\begin{Prop} \label{prop-dim-R/I}
Let $I$ be an $\m$-primary ideal. Suppose $r > 0$ is such that $\m^r \subset I$. Then for any $m \geq r$ we have:
$$\dim_\k(R/I) = \dim_\k(A^{(m)} / I^{(m)}).$$
\end{Prop}
\begin{proof}
The inclusion $A^{(m)} \subset R$ gives $\k$-linear maps from $A^{(m)}$ to $R/\m^m$ and to $R/I$. Note that the image of $A^{(m)}$ in the quotient space $R/ \m^m$ spans $R / \m^m$. It follows that the $\k$-linear map $A^{(m)} \to R/ \m^m \to R / I$ is surjective. Clearly the kernel of this map is $I^{(m)} = A^{(m)} \cap I$. Thus $A^{(m)} / I^{(m)} \cong R/I$ as 
vector spaces. This finishes the proof.
\end{proof}

The following lemma relates the product of $\m$-primary ideals with the product of subspaces. 
\begin{Lem} \label{lem-IJ}
Let $I, J$ be $\m$-primary ideals. Suppose $r_1, r_2 > 0$ are such that $\m^{r_1} \subset I$ and 
$\m^{r_2} \subset J$. Then for $p, q \geq r_1 + r_2$ we have: 
$$I^{(p)} J^{(q)} = (IJ)^{(p+q)}.$$
\end{Lem}
\begin{proof}
The subspace $I^{(p)}$ is the sum of the subspace $I^{(r_1-1)}$ (consisting of restrictions of polynomials of degree $< r_1$ in $I$) and the sum $\sum_{r_1 \leq d \leq p}P^{(d)}$ (of the subspaces $P^{(d)}$ consisting of restrictions of homogeneous polynomials of degree $d$). Similarly $J^{(q)} = J^{(r_2-1)} +
\sum_{r_2 \leq k \leq p} P^{(k)}$. Thus $I^{(p)}J^{(q)}$ contains the sum of the subspaces $P^{(m)}$ of restrictions of homogeneous polynomials of degree $m$ where $r_1 + r_2 \leq m \leq p+q$ and some subspace $L$ of restrictions of polynomials of degree $< r_1 + r_2$. Note that if we increase $p$ or $q$ the space $L$ does not change as this only adds restrictions of some polynomials of degree not smaller than $r_1 + r_2$ to the product. 
It follows that the space $(IJ)^{(p+q)}$ also is equal to $L + \sum_{r_1 + r_2 \leq k \leq p + q} P^{(m)}$ which proves the lemma.
\end{proof}

As in Section \ref{sec-int-index} let $\G(X)$ denote the Grothendieck group of finite dimensional subspaces of rational functions on $X$ 
(see Definition \ref{def-G(X)}). The following readily follow from Lemma \ref{lem-IJ}.
\begin{Cor} \label{cor-k-m-large}
Let $I$ be an $\m$-primary ideal in the local ring $R$. Suppose $r>0$ is such that $\m^r \subset I$. Then for $k, m \geq r$ we have:
$$I^{(m)} / A^{(m)} = I^{(k)} / A^{(k)},$$ as elements of the Grothendieck group $\G(X)$.
\end{Cor}

\begin{Cor} \label{cor-int-index-I/A}
Let $I_1, \ldots, I_n$ be $\m$-primary ideals in the local ring $R$. Then for sufficiently large integers $k_1, \ldots, k_n$ and $m$ we have 
$$[I_1^{(k_1)} / A^{(k_1)}, \ldots, I_n^{(k_n)} / A^{(k_n)}] = [I_1^{(m)} / A^{(m)}, \ldots, I_n^{(m)} / A^{(m)}],$$
where $[\cdot, \ldots, \cdot]$ denotes the intersection index in the Grothendieck group $\G(X)$ (see Section \ref{sec-int-index}). 
\end{Cor}

Finally we have the following natural homomorphism from the Grothendieck group of $\m$-primary ideals
into the Grothendieck group $\G(X)$ of subspaces of rational functions.
Let $I \in \K_{ideal}(R)$ and let $r > 0$ be such that $\m^r \subset I$ and let $m \geq r$.
\begin{Th} \label{Groth-homomorphism}
The map $\iota: I \mapsto I^{(m)} / A^{(m)}$ gives a homomorphism of groups $\iota: \G_{ideal}(R) \to \G(X)$. Moreover, the map $\iota$ is independent of the choice of $m \geq r$.
\end{Th}
\begin{proof}
By Corollary \ref{cor-k-m-large} the map is well-defined i.e. independent of $m \geq r$. Also by Lemma \ref{lem-IJ} it is a homomorphism.
\end{proof}

\section{Multiplicity of a system of equations at a root} \label{sec-multi-root}
Let $X$ be an algebraic variety of pure dimension $n$
and let $R = \mathcal{O}_{X, a}$ denote the local ring of $X$ at some point $a$. 

Let $f_1, \ldots, f_n \in R$ be regular functions at $a$ and assume $a$ is an isolated solution of the system
\begin{equation} \label{equ-system-f}
f_1(x) = \cdots = f_n(x) = 0.
\end{equation}
In this section we review the classical notion of {\it multiplicity} of $a$ as a root of the system \eqref{equ-system-f}, 
in other words, the {\it intersection multiplicity} of the hypersurfaces $H_i = \{f_i(x) = 0\}$, $i=1, \ldots, n$, 
at the point $a$.
Intuitively, the multiplicity of a root $a$ of a system \eqref{equ-system-f}
is the number of roots which are born around $a$ if we slightly perturb the system. 

When the ground field $\k$ is the field of complex numbers $\c$ one can give the following topological
definition for the multiplicity.
\begin{Def}[Topological definition of multiplicity] \label{def-multi-top}
Let $U$ be a sufficiently small neighborhood of $a$ (in the usual topology of $X$) 
with real analytic boundary $\partial U$
and consider the map $F: \partial U \to S^{2n-1}$ given by
$$F(x) = (\frac{f_1(x)}{|{\bf f}(x)|}, \ldots, \frac{f_n(x)}{|{\bf f}(x)|}),$$ where 
$|{\bf f}(x)| = (|f_1(x)|^2 + \cdots + |f_n(x)|^2)^{1/2}$. The 
{\it multiplicity} $e_a(f_1, \ldots, f_n)$ of the root $a$ of the system \eqref{equ-system-f} is equal to the 
mapping degree of $F$. 
\end{Def}

One knows that the multiplicty defined in Definition \ref{def-multi-top} satisfies 
the following properties.
\begin{itemize} 
\item[(i)] Let $M$ be an $n \times n$ matrix whose entries are regular functions at $a$ and put 
$(g_1, \ldots, g_n)^t = M (f_1, \ldots, f_n)^t$. 
Then $e_a(f_1, \ldots, f_n) = e_a(g_1, \ldots, g_n)$ if and only if $\det(M)(a) \neq 0$. 

\item[(ii)] Let $Y$ be a projective subvariety of dimension $n$ in a projective space $\p^N$. Let $H$ be a plane of codimension $n$ in $\p^N$ and assume that $Y \cap H$ is finite. Then the number of points in $Y \cap H$ counted  with multiplicity is equal to $\deg(Y)$.

\item[(iii)] Let $I_{\bf f}$ be the ideal generated by $f_1, \ldots, f_n$. Then 
\begin{equation} \label{equ-multi-system-Samuel}
e_a(f_1, \ldots, f_n) = n!~ \lim_{k \to \infty} \frac{\dim_\k (R / I_{\bf f}^k)}{k^n}.
\end{equation}
Moreover, if $a$ is a smooth point of $X$ then $e_o(f_1, \ldots, f_n) = \dim_{\k}(R / I_{\bf f})$.
\end{itemize}

For a general algebraically closed field $\k$, one takes the property (iii) as definition.
\begin{Def}[Algebraic definition of multiplicity of a system of equations] 
\label{def-multi-system}
One defines the {\it multiplicity $e_a(f_1, \ldots, f_n)$ 
of the system ${\bf f} = (f_1, \ldots, f_n)$ at $a$} by the formula \eqref{equ-multi-system-Samuel}.
If $a$ is a smooth point then the multiplicity can be defined as $\dim_{\k}(R/ I_{\bf f})$ (it is known that for a smooth point $a$ the two definitions coincide).
\end{Def}

The following is well-known.
\begin{Th} \label{th-properties-multi-system}
The multiplicity of a system of equations in Definition \ref{def-multi-system}
satisfies the properties (i) and (ii) above.
\end{Th}

Note that we can replace the variety $X$ with an affine neighborhood of the point $a$. Thus without loss of generality we can assume $X$ is affine. We will
fix an embedding of $X$ into some affine space $\A^N$ and assume that the point $a$ is the origin $o \in \A^N$.

Consider a system of equations $f_1(x) = \cdots = f_n(x) = 0$ and assume that $o$ is an isolated root of this system. Suppose each $f_i$ is the
restriction of a rational function $P_i/Q_i$ to $X$ where $P_i$, $Q_i$ are polynomials in $\AA = \k[x_1, \ldots, x_N]$ and $Q_i(o)$ is not equal to $0$ for all $i$.

\begin{Prop} \label{prop-multi-reg-vs-poly}
The multiplicity of $o$ as an isolated root of the system $f_1(x) = \ldots, f_n(x) = 0$ is equal the multiplicity of $o$ as a root of 
the system $P_1(x) = \ldots = P_n(x) = 0$. In particular $o$ is an isolated root of this system too.
\end{Prop}
\begin{proof}
Follows directly from property (i) of the multiplicity.
\end{proof}
Thus the study of the multiplicities of systems of equations is reduced to the case where the variety is affine and the functions are restrictions of polynomials.

\section{Systems of equations with no roots at infinity} \label{sec-no-sol-infinity}
Let $n \leq N$ and consider a system of $n$ polynomial equations in the variables $x = (x_1, \ldots, x_N) \in \A^N$, 
\begin{equation} \label{equ-poly-system}
P_1(x) = \cdots = P_n(x) = 0.
\end{equation}
Let $W \subset \A^N$ be the algebraic set defined by \eqref{equ-poly-system}. 
Consider the usual projective completion of the affine space $\A^N \subset \p^N$ and let $L_\infty = \p^N \setminus \A^N$ 
be the infinite hyperplane. Let $Z$ be the intersection in $\p^N$ of $L_\infty$ and $\overline{W}$.

Let $\AA$ denote the polynomial algebra $\k[x_1, \ldots, x_N]$. 
Let $\phi: X \hookrightarrow \A^N$ be an embedding of an affine variety $X$ of pure dimension $n$ 
into the affine space $\A^N$ and let $\phi^*: \AA \to \k[X]$ denote the corresponding restriction map on the coordinate rings.

\begin{Def} \label{def-system-no-sol-infinity}
We say that a point $a \in \p^N$ is a {\it root at infinity of the system \eqref{equ-poly-system} on $X$} if 
$a \in \overline{\phi(X)} \cap Z$.
\end{Def}

Let $m > 0$ be the maximum degree of the $P_i$.
To the system \eqref{equ-poly-system} one associates the system 
\begin{equation} \label{equ-poly-homo-system}
P_1^{(m)}(x) = \cdots = P_n^{(m)}(x) = 0,
\end{equation}
where $P_i^{(m)}$ denotes the degree $m$ homogeneous component of $P_i$. The following is easy to verify.
\begin{Prop} \label{prop-root-at-infinity}
If $a \in \overline{\phi(X)} \cap Z$ is a root at infinity of the system \eqref{equ-poly-system} then $a$ is a root at infinity on $X$ for the system
\eqref{equ-poly-homo-system}.
\end{Prop}

\begin{Rem} \label{rem-root-at-infinity}
Note that if $a$ is a root at infinity on $X$ for the system \eqref{equ-poly-homo-system} then it is not necessarily a root 
at infinity on $X$ for the system \eqref{equ-poly-system}.
\end{Rem}

As before let $\AA^{(m)}$ and $\PP^{(m)}$ 
denote the vector space of polynomials in $\AA$ of degree less than or equal to $m$, and the vector space of 
homogeneous polynomials of degree $m$ respectively.  
For each $m > 0$ we have the Kodaira embedding (in this case also called the Veronese embedding) 
$\Phi_m : \p^N \to \p(\AA^{(m)*})$ where $\AA^{(m)*}$ is the dual space of $\AA^{(m)}$. 
Let $Y_m$ be the closure of the image of $X$ under the map $\Phi_m \circ \phi$. 
Also as before let $A^{(m)} = \phi^*(\AA^{(m)})$, that is, $A^{(m)}$ is the vector space of regular functions on $X$ which are 
restrictions of polynomial functions on $\A^N$ of degree at most $m$.
\begin{Th} \label{th-system-no-sol-infinity}
Let $L$ be a vector subspace of $A^{(m)}$ which contains the restrictions of all homogeneous degree 
$m$ polynomials i.e. $\PP^{(m)}$. Consider a system of polynomials from $L$ as in \eqref{equ-poly-system}. 
\begin{itemize}
\item[(1)] There is a nonempty Zariski open subset in $L \times \cdots \times L$ 
such that any system in this open subset has no roots at infinity on $X$. 
\item[(2)] If the system has no roots at infinity on $X$ then the number of roots of the system on $X$, counted with multiplicity, is 
equal to $\deg(Y_m) = [A^{(m)}, \ldots, A^{(m)}]$.
\end{itemize}
\end{Th}
\begin{proof}
Follows from Proposition \ref{prop-root-at-infinity} and Theorem \ref{th-properties-multi-system}(ii).
\end{proof}

Finally we show below (Proposition \ref{prop-isolated-roots-no-sol-infinity}) that we can 'perturb' a system such that the multiplicity at $a$ is unchanged,  the new system has no roots at infinity and and all its roots except the origin are simple.

Without loss of generality we assume $a = o$ is the origin.
Let $f_1, \ldots, f_n \in \m$ and let $I$ be the ideal generated by the $f_i$ in $R = \mathcal{O}_{X, o}$. 
Assume that $o$ is an isolated solution of the system $f_1(x) = \cdots = f_n(x) = 0$, i.e. the ideal $I$ is $\m$-primary. Then there is $r>0$ such that $\m^r \subset I$.

\begin{Lem} \label{lem-Nakayama}
Let $h_1, \ldots, h_n \in \m^{m}$ where $m > r$, then the ideal generated by $f_1 + h_1, \ldots, f_n + h_n$ coincides with $I$, the ideal generated by the $f_i$.
\end{Lem}
\begin{proof}
Let $J$ denote the ideal generated by the $f_i + h_i$. Then $I = J + \m^{m}$ and hence $I = J + \m I$. By 
Nakayama's lemma we have $I = J$ as claimed.
\end{proof}

\begin{Prop} \label{prop-isolated-roots-no-sol-infinity}
Let $f_1, \ldots, f_n \in \m$ be as above. Then we can find $g_1, \ldots, g_n \in \m$ such that:
\begin{itemize}
\item[(1)] The $g_i$ generate the same ideal as the $f_i$ in the local ring $R$ and hence $e_o(g_1, \ldots, g_n) = e_o(f_1, \ldots, f_n)$.
\item[(2)] The system $g_1(x) = \ldots = g_n(x) = 0$ has no roots at infinity on $X$, moreover, the roots of this system in $X$ except the root $a$, lie in the smooth locus of $X$ and are isolated and simple.
\end{itemize}
\end{Prop}
\begin{proof}
Let $m > r$, then by Lemma \ref{lem-Nakayama} for any $h_1, \ldots, h_n \in \m^m$, the ideal generated by the 
$g_i = f_i + h_i$ is the same as $I$ and hence $e_o(g_1, \ldots, g_n) = e_o(f_1, \ldots, f_n)$. Also by 
Theorem \ref{th-system-no-sol-infinity} (1), if the $h_i$ are generic then the system $g_1(x) = \cdots = g_n(x) = 0$ has no roots at infinity on $X$. Since the variety defined by the system is closed and does not intersect the hyperplane at infinity it follows that it is a finite set of points, i.e. all the roots are isolated. Now consider the subspace $L_m$ spanned by all the $f_i$ and all functions which are the restrictions of homogeneous polynomials of degree $m$, i.e. $\PP^{(m)}$. By Theorem \ref{th-int-index-well-def} for a generic system in $L_m$, all the roots lie in the smooth locus of $X$ and are simple.
\end{proof}

\section{Intersection indices of subspaces associated to $\m$-primary ideals} \label{sec-int-index-subspace-ideal}
The following theorem plays an important role for us. An analogous relation holds for the mixed 
volume of convex bodies (see Section \ref{sec-convex-geo}). As before $X$ is an affine variety of pure dimension $n$ embedded in 
an affine space $\A^N$. We assume the origin $o$ is in $X$ and put $R = \mathcal{O}_{X, o}$.
\begin{Th}[Orthogonality relations for intersection indices] \label{th-orth-relation}
Let $0 < p < n$ and let $I_1, \ldots, I_p$ be $\m$-primary ideals in the local ring $R$. Then for sufficiently large $m > 0$ we have
\begin{itemize} 
\item[(1)] $$[I_1^{(m)}, \ldots, I_p^{(m)}, A^{(m)}, \ldots, A^{(m)}] = [A^{(m)}, \ldots, A^{(m)}].$$
(Recall that the intersection index $[A^{(m)}, \ldots, A^{(m)}]$ is the degree of the projective variety $Y_m$ 
(Theorem \ref{th-system-no-sol-infinity}(2).)
\item[(2)] It follows that for sufficiently large $m > 0$ we have
$$[I_1^{(m)} / A^{(m)}, \ldots, I_p^{(m)} / A^{(m)}, A^{(m)}, \ldots, A^{(m)}] = 0.$$
\end{itemize}
\end{Th}
\begin{proof}
(1) Let $r$ be such that $\m^r \subset I_i$ for all $i$ and let $m > r$. Note that this implies that
the base locus of each $I_i^{(m)}$ consists of the origin only.
By Theorem \ref{th-system-no-sol-infinity}(1) one then knows that a generic system 
$(P_1, \ldots, P_p, \ldots, P_n) \in I_1^{(m)} \times \cdots \times I_p^{(m)} 
\times A^{(m)} \times \cdots \times A^{(m)}$ has no solutions at infinity on $X$ and also, by Proposition \ref{prop-isolated-roots-no-sol-infinity}, all its solutions in 
$X \setminus \{o\}$ are simple (i.e. have multiplicity $1$). On the other hand, by 
Theorem \ref{th-system-no-sol-infinity}(2), 
the number of solutions of the system \eqref{equ-poly-system} counted with multiplicity is equal to 
$[A^{(m)}, \ldots, A^{(m)}]$. But since the origin is not a solution of $P_n(x) = 0$ for a generic $P_n \in A^{(m)}$ we conclude that $[I_1^{(m)}, \ldots, I_p^{(m)}, A^{(m)}, \ldots, A^{(m)}] = [A^{(m)}, \ldots, A^{(m)}]$ as required. (2) Immediately follows from (1) and multi-additivity of the intersection index.
\end{proof}

\section{Mixed multiplicity of ideals} \label{sec-mixed-multi-ideal}
As usual let $X$ be a variety of pure dimension $n$ and let $R = \mathcal{O}_{X, a}$ denote the local ring of $X$ at some point $a$, 
with maximal ideal $\m$. We assume $X$ is affine, embedded in some affine space $\A^N$ and $a = o$ is the origin. 
In this section we give a definition for the mixed multiplicity of an 
$n$-tuple of $\m$-primary ideals, as the
multiplicity of a generic system of equations from this $n$-tuple of ideals (see Definition \ref{def-multi-system}).
It can be regarded as an analogue of the intersection index of an $n$-tuple of subspaces of rational functions versus 
the number of solutions of a system of equations. The next theorem and definition (Theorem \ref{th-multi-ideals} and Definition \ref{def-multi-ideals}) are analogues of Theorem \ref{th-int-index-well-def} and Definition 
\ref{def-int-index} respectively. We will see later that definition of mixed multiplicity of ideals coincides with the usual definition, in commutative algebra, as the polarization of the multiplicity polynomial (Theorem \ref{th-Samuel-vs-ours}).

We will prove the following later (Theorem \ref{th-multi-int-index}).
\begin{Th} \label{th-multi-ideals}
Let $I_1, \ldots, I_n$ be $\m$-primary ideals in $R$. (1) For sufficiently large $m>0$, there exists a nonempty Zariski open subset $\U_m \subset I^{(m)}_1 \times \cdots \times I^{(m)}_n$ such that for any 
${\bf f} = (f_1, \ldots, f_n) \in \U_m$, the {\it multiplicity} $e_o(f_1, \ldots, f_n)$ of the system 
${\bf f}$ at the origin $o$ is the same. (2) Moreover, this is the smallest multiplicity among all the systems $(g_1, \ldots, g_n) \in I^{(m)}_1 \times \cdots \times I^{(m)}_n$ with isolated root at $o$.
\end{Th}

\begin{Def} \label{def-multi-ideals}
We call the multiplicity of a generic system in $I^{(m)}_1 \times \cdots \times I^{(m)}_n$ in Theorem \ref{th-multi-ideals}, 
the {\it mixed multiplicity} $e(I_1, \ldots, I_n)$ of the ideals $I_i$. 
For  a single $\m$-primary ideal $I$, the multiplicty $e(I, \ldots, I)$ is usually denoted by $e(I)$.
\end{Def}


Our goal in the paper is to deduce some basic properties of the (local) mixed multiplicity of $\m$-primary ideals by reducing them to statements about (global) intersection index of subspaces of rational functions. 

\section{Intersection index of subspaces and mixed multiplicity of $\m$-primary ideals} \label{sec-int-index-mixed-multi}
The next theorem and its corollary are our main results that relate the mixed multiplicity of $\m$-primary ideals and intersection index of subspaces of rational functions.  
\begin{Th} \label{th-multi-int-index}
Let $I_1, \ldots, I_n$ be $\m$-primary ideals in the local ring $R$. Then for any sufficiently large $m > 0$, and  
any generic system $(f_1, \ldots, f_n) \in (I^{(m)}_1, \ldots, I^{(m)}_n)$, we have 
$e_o(f_1, \ldots, f_n) = [A^{(m)}, \ldots, A^{(m)}] - [I_1^{(m)}, \ldots, I_n^{(m)}]$. In other words,
\begin{equation} \label{equ-multi-int-index}
e(I_1, \ldots, I_n) = [A^{(m)}, \ldots, A^{(m)}] - [I_1^{(m)}, \ldots, I_n^{(m)}].
\end{equation}
\end{Th}
\begin{proof}
Let $m > 0$ be sufficiently large so that the base loci of any $I^{(m)}_i$ is $\{o\}$. Take a generic system 
(as in Proposition \ref{prop-isolated-roots-no-sol-infinity}) such that: (1) the system 
$f_1(x) = \cdots = f_n(x) = 0$ has not roots at infinity on $X$, and hence all its roots are isolated, 
and (3) all the roots of the system, except the origin, are simple (i.e. have multiplicity $1$). Then by Theorem 
\ref{th-system-no-sol-infinity} the number of roots of the system counted with multiplicity is $[A^{(m)}, \ldots, A^{(m)}]$. On the other hand, the number of roots of the system in $X \setminus \{o\}$ is equal to $[I^{(m)}_1, \ldots, I^{(m)}_n]$ (note that all the roots in $X \setminus \{o\}$ are simple). Thus we obtain $e(I_1, \ldots, I_n) = e_o(f_1, \ldots, f_n) =  [A^{(m)}, \ldots, A^{(m)}] - [I_1^{(m)}, \ldots, I_n^{(m)}]$ as required.
\end{proof}

From Theorem \ref{th-multi-int-index} we obtain the following corollary that expresses the mixed multiplicity of an $n$-tuple of 
$\m$-primary ideals $(I_1, \ldots, I_n)$ as the intersection index of certain classes in the 
Grothendieck group $\G(X)$ associated to $I_1, \ldots, I_n$.
\begin{Cor} \label{cor-multi-int-index}
Let $I_1, \ldots, I_n$ be $\m$-primary ideals in the local ring $R$ then for sufficiently large $m > 0$ we have:
\begin{equation} \label{equ-multi-int-index-I/A}
e(I_1, \ldots, I_n) =  -[I_1^{(m)} / A^{(m)}, \ldots, I_n^{(m)} / A^{(m)}].
\end{equation}
(Recall that $I \mapsto I^{(m)} / A^{(m)}$ is the homomorphism of groups $\iota$ in Theorem 
\ref{Groth-homomorphism}.)
\end{Cor}
\begin{proof}
By multi-additivity of intersection index and Theorem \ref{th-orth-relation} we have:
\begin{eqnarray*}
[I_1^{(m)} / A^{(m)}, \ldots, I_n^{(m)} / A^{(m)}] 
&=& [I_1^{(m)}, \ldots, I_n^{(m)}] + (\sum_{k=1}^{n} (-1)^{k} {n \choose k}) 
[A^{(m)}, \ldots, A^{(m)}], \cr
&=&  [I_1^{(m)}, \ldots, I_n^{(m)}] - [A^{(m)}, \ldots, A^{(m)}],\cr
\end{eqnarray*}
which is equal to $-e(I_1, \ldots, I_n)$ by Theorem \ref{th-multi-int-index}.
\end{proof}

\begin{Cor}[Mixed multiplicity and restriction] \label{cor-multi-rest}
Let $I_1, \ldots, I_n$ be $\m$-primary ideals in $R$. Let $m > 0$ be a sufficiently large integer. For any 
$1 \leq k < n$ let $(f_{k+1}, \ldots, f_n) \in I_{k+1}^{(m)} \times \cdots \times I_n^{(m)}$ be a generic $(n-k)$-tuple of functions and let $Y$ be 
the subvariety defined in a neighborhood of $o$ by the system of equations $f_{k+1}(x) = \cdots = f_n(x) = 0$.
Then: 
$$e({I_1}_{|Y}, \ldots, {I_k}_{|Y}) = e(I_1, \ldots, I_n).$$ 
\end{Cor}
\begin{proof}
Follows from Theorem \ref{th-multi-int-index} and Theorem \ref{th-int-index-rest}.
\end{proof}

\section{Proofs of basic properties of mixed multiplicity of ideals} \label{sec-proof-prop-mixed-multi}
In this section we show that the formula \eqref{equ-multi-int-index}, which represents the mixed multiplicity as an intersection index 
in the Grothendieck group of finite dimensional subspaces of rational functions, readily implies some basic properties of the mixed multiplicity of $\m$-primary ideals. In particular we show that our notion of mixed multiplicity (Definition \ref{def-multi-ideals}) 
coincides with classical Samuel's notion of mixed multiplicity. 
\begin{Th}[Multi-additivity of mixed multiplicity] \label{th-mixed-multi-additive}
The mixed multiplicity is multi-additive. That is, 
if $I_1', I_1'', I_2, \ldots, I_n$ are $\m$-primary ideals then:
$$e(I_1'I_1'', I_2, \ldots, I_n) = e(I_1', I_2, \ldots, I_n)+e(I_1'', I_2, \ldots, I_n).$$
\end{Th}
\begin{proof}
Follows directly from Theorem \ref{Groth-homomorphism} and Corollary \ref{cor-multi-int-index}.
\end{proof}

The next corollary is immediate from Theorem \ref{th-mixed-multi-additive}.
\begin{Cor}[Polynomiality of multiplicity] \label{cor-multi-polynomial}
Let $I_1, \ldots, I_n$ be $\m$-primary ideals in $R$. Then the multiplicity 
$$e(I_1^{k_1} \cdots I_n^{k_n})$$ 
is a homogeneous polynomial of degree $n$ in $k_1, \ldots, k_n$.
\end{Cor}

Similar to the notion of Hilbert function of a projective subvariety (in a projective space), 
one defines the Hilbert-Samuel function of an $\m$-primary ideal.
For an $\m$-primary ideal $I$, the {\it Hilbert-Samuel function}  $H_I(k)$ is defined by: $$H_I(k) = \dim_\k(R / I^k).$$
Note that since $I$ is $\m$-primary, the vector spaces $R/I^k$ are all finite dimensional.
It is well known that this function is a polynomial of degree $n$ 
for sufficiently large values of $k$. This allows one to define the {\it Samuel multiplicity} of $I$ 
algebraically as:
$$e(I) = n!~\lim_{k \to \infty} \frac{H_I(k)}{k^n}.$$
Take $\m$-primary ideals $I_1, \ldots, I_n$. 
One can show that the function $e(I_1^{k_1} \cdots I_n^{k_n})$ is a polynomial of degree $n$ in the $k_i$.  
The mixed Samuel multiplicity $e(I_1, \ldots, I_n)$ of $\m$-primary ideals $I_1, \ldots, I_n$ is defined to be the coefficient of $k_1 \dots k_n$ 
in this polynomial divided by $n!$.

Below (using Hilbert's theorem on the degree of a projective variety and our intersection theory of subspaces of rational functions) we prove that our definition 
of mixed multiplicity coincides with that of Samuel.

We state a version of Hilbert's theorem formulated in terms of
intersection index of finite dimensional subspaces of rational functions:
\begin{Th}[A version of Hilbert's theorem] \label{th-Hilbert}
Let $\Phi_L: X \ratmap \p(L^*)$ be the Kodaira map of a subspace $L \in \K(X)$ and 
let $Y_L$ denote the closure of the image of $\Phi_L$ (see Definition \ref{def-Kodaira-map} and Proposition \ref{prop-int-index-degree}). 
We assume that $\Phi_L$ is a birational isomorphism between $X$ and $Y_L$. Then for sufficiently large 
$k$, the Hilbert function $H_{Y_L}(k) = \dim_\k(L^k)$ is a polynomial of degree $n = \dim(X)$ and:
$$\deg(Y_L) = n!~\lim_{k \to \infty} \frac{H_{Y_L}(k)}{k^n}.$$ 
Here $\deg(Y_L)$ denotes the degree of $Y_L$ as a subvariety of the projective space $\p(L^*)$.
\end{Th}

\begin{Th} \label{th-Samuel-vs-ours}
The Samuel notion of mixed multiplicity coincides with the notion of mixed multiplicity in Definition \ref{def-multi-ideals}.
\end{Th}
\begin{proof}
We need the following:
\begin{Lem} \label{lem-I(m)-embedding}
If $m>0$ is sufficiently large then the Kodaira map $\Phi_{I^{(m)}}$ gives a birational isomorphism between $X$ and $Y_{I^{(m)}}$, the closure of 
image of $X$ under $\Phi_{I^{(m)}}$. Moreover the restriction of $\Phi_{I^{(m)*}}$ to $X \setminus \{0\}$ 
gives an embedding of $X \setminus \{0\}$ to $\p(I^{(m)*})$.
\end{Lem}
\begin{proof}
It suffices to show that for sufficiently large $m > 0$, $\Phi_{I^{(m)}}$ is one-to-one, i.e. functions in $I^{(m)}$ separate generic points of $X$. Let $r > 0$ be such that $\m^r \subset I$. One verifies that the subspace $$\M = \{ p \in \k[x_1, \ldots, x_N] \mid r \leq \deg(p) \leq 2r \},$$ generates the algebra $\AA_{\geq r}$ of polynomials all whose terms have degree bigger than or equal to $r$. Hence the image of $\M$ in $R$ generates the algebra $A_{\geq r}$, the image of $\AA_{\geq r}$ in $R$. But functions in $A_{\geq r}$ separate generic elements of $X$. This finishes the proof.    
\end{proof}
By Lemma \ref{lem-I(m)-embedding} we can choose $m > 0$ large enough so that 
the Kodaira map $\Phi_{I^{(m)}}$ gives a birational embedding of $X$ in the projective space $\p({I^{(m)}}^*)$. Now 
by Proposition \ref{prop-dim-R/I}, $\dim_\k(R / I^k) = \dim_\k(A^{(km)} / (I^k)^{(km)})$. Moreover, 
$A^{(mk)} = (A^{(m)})^k$ and $(I^k)^{(km)} = (I^{(m)})^k$. 
Thus, $$\dim_\k(R / I^k) = \dim_\k((A^{(m)})^k / (I^{(m)})^k) = \dim_\k((A^{(m)})^k) - \dim_\k((I^{(m)})^k).$$
The theorem now follows from Theorem \ref{th-Hilbert} and Theorem \ref{th-multi-int-index}.
\end{proof}

\section{Alexandrov-Fenchel inequality for mixed multiplicities} \label{sec-AF-ideal}
In this section we prove a (reverse) Alexandrov-Fenchel inequality for mixed multiplicities of $\m$-primary ideals.
As usual $R$ is the local ring $\mathcal{O}_{X, a}$ of a point $a$ on a (possibly reducible) algebraic variety $X$ of dimension $n$.
Without loss of generality, we assume all the irreducible components of $X$ have the same dimension.
\begin{Th}[Alexandrov-Fenchel inequality for mixed multiplicities] \label{th-AF-multi}
Let $I_1, I_2, \ldots, I_n$ be $\m$-primary ideals in the local ring $R$. The following inequality holds between the mixed multiplicities:
$$e(I_1, I_1, I_3, \ldots, I_n) e(I_2, I_2, I_3, \ldots, I_n) \geq e(I_1, I_2, I_3, \ldots, I_n)^2.$$
\end{Th}

We prove this inequality using the Hodge inequality for intersection indices of subspaces (Theorem \ref{th-Hodge}).
We need the following simple lemma from linear algebra which we state without proof.
\begin{Lem} \label{lem-neg-def}
Let $V$ be a finite dimensional vector space over $\r$ equipped with a symmetric bilinear form $b( \cdot, \cdot)$ 
with only one positive eigenvalue i.e. the rest of eigenvalues are negative or zero.
Fix a vector $v \in V$ with $b(v, v) > 0$. Let $W = \{ w \in V \mid b(v, w) = 0\}$ be the orthogonal subspace of $v$. Then the bilinear form 
$b$ restricted to the subspace $W$ is negative semi-definite. 
\end{Lem}
\begin{proof}[Proof of Theorem \ref{th-AF-multi}]
Let $m>0$ be a sufficiently large integer and 
let $Y$ be as in Corollary \ref{cor-multi-rest}. By Theorem \ref{th-int-index-rest} all the irreducible components of $Y$ have dimension $2$.
Let $Y_i$ be an irreducible component of $Y$ and let $V_i$ denote the vector space spanned by 
$A^{(m)}_{|Y_i}, {I_1}^{(m)}_{|Y_i}, {I_2}^{(m)}_{|Y_i}$.
By the Hodge inequality (Theorem \ref{th-Hodge}) we know that the intersection index $[\cdot, \cdot]$ on $Y_i$, regarded as a bilinear form on $V_i$,
has one positive eigenvalue and all other eigenvalues are negative or zero. From Lemma \ref{lem-neg-def} the intersection 
index on $Y_i$ restricted to the orthogonal space of $A^{(m)}_{|Y_i}$ is negative semi-definite. 
Since the sum of negative semi-definite bilinear forms is again negative semi-definite, it follows that 
the intersection index on $Y$ restricted to the orthogonal space of $A^{(m)}_{|Y}$ is negative semi-definite. Hence it satisfies the 
Cauchy-Schwartz inequality. We note that by Theorem \ref{th-orth-relation}, the elements ${I_j}^{(m)}_{|Y} / A^{(m)}_{|Y}$, $j=1,2$, 
lie in the orthogonal space to $A^{(m)}_{|Y}$. 
The claim now follows from Corollary \ref{cor-multi-int-index} and Corollary \ref{cor-multi-rest}.
\end{proof}

\section{Applications to convex geometry} \label{sec-convex-geo}
In this section we use Theorem \ref{th-AF-multi} (Alexandrov-Fenchel inequality for mixed multiplicities) to give an alternative proof of the Alexandrov-Fenchel inequality for covolumes of convex bodies proved in \cite{Askold-Vladlen}.

Let $C$ be a closed strictly convex cone of dimension $n$ with apex at the origin in the Euclidean space $\r^n$ 
(strictly convex cone means it is convex and does not contain any lines through the origin).
\begin{Def}[Convex region in a cone]
We call a subset $\Gamma \subset C$ a {\it $C$-convex region}, if $\Gamma$ is closed and 
convex and moreover for any $x \in \Gamma$ and $y \in C$ we have $x+y \in \Gamma$. 
One also refers to the set $C \setminus \Gamma$ as a {\it coconvex set} (with respect to the cone $C$). 
We call a $C$-convex region {\it cobounded} if the completment $C \setminus \Gamma$ is bounded. 
We call the volume of $C \setminus \Gamma$ the {\it covolume} of the convex region $\Gamma$ and denote it by $\covol(\Gamma)$.
\end{Def}

\begin{Rem} \label{rem-convex-region-ideal}
The notion of a $C$-convex region in a cone $C$ is an analogue of an ideal in a ring. The notion of cobounded convex region is an analogue of an $\m$-primary ideal.
\end{Rem}

It is easy to verify the following: 

\begin{Prop} \label{prop-convex-region-closed}
the collection of $C$-convex regions is closed under addition and multiplication by a positive scalar. That is, if $\Gamma_1, \Gamma_2 \subset C$ are $C$-convex regions and $\lambda_1, \lambda_2 > 0$ then 
$$\lambda_1\Gamma_1 + \lambda_2 \Gamma_2 = \{ \lambda_1 x_1 + \lambda_2 x_2 
\mid x_1 \in \Gamma_1,~x_2 \in \Gamma_2\}$$
is also a $C$-convex region. Moreover, if $\Gamma_1$, $\Gamma_2$ are cobounded then $\lambda_1 \Gamma_1 + \lambda_2 \Gamma_2$ is also cobounded.
\end{Prop}

We will prove the following using multi-additivity of the intersection index.
\begin{Th}[Covolume is polynomial] \label{th-covol-poly}
Let $\Gamma_1, \ldots, \Gamma_n$ be cobounded $C$-convex regions. Then the function
\begin{equation} \label{equ-covol-poly}
P(\lambda_1, \ldots, \lambda_n) = \covol(\lambda_1 \Gamma_1 + \cdots + \lambda_n \Gamma_n),
\end{equation}
is a homogeneous polynomial of degree $n$ in the $\lambda_i$. 
\end{Th}

Since the covolume is a polynomial in the space of cobounded $C$-convex regions, it can be extended to a multi-linear function. More precisely, let $\mathcal{C}$ denote the collection of all cobounded $C$-convex regions. Then there exists a unique function $V: \mathcal{C}^n \to \r$ such that: (1) $V$ is linear in each argument, (2) for any cobounded $C$-convex region $\Gamma$ we have $V(\Gamma, \ldots, \Gamma) = \covol(\Gamma)$. For $C$-convex regions 
$\Gamma_1, \ldots, \Gamma_n$ we call $V(\Gamma_1, \ldots, \Gamma_n)$ the {\it mixed covolume} of the $\Gamma_i$.

Similar to the mixed volume of convex bodies, the mixed covolume also satisfies an Alexandrov-Fenchel inequality. We give a proof of this inequality in this section, using an analogous inequality proved earlier for mixed multiplicities of ideals (Theorem \ref{th-AF-multi}). 

\begin{Th}[Alexandrov-Fenchel inequality for mixed covolume] \label{th-AF-covolume}
Let $\Gamma_1, \Gamma_2, \ldots, \Gamma_n$ be cobounded $C$-convex regions. Then we have:
\begin{equation} \label{equ-AF-covol}
V(\Gamma_1, \Gamma_1, \Gamma_3, \ldots, \Gamma_n) V(\Gamma_2, \Gamma_2, \Gamma_3, \ldots, \Gamma_n) \geq V(\Gamma_1, \Gamma_2, \ldots, \Gamma_n)^2.
\end{equation}
\end{Th}

To prove Theorems \ref{th-covol-poly} and \ref{th-AF-covolume} 
first we consider the case where $C$ is a rational polyhedral cone. To such a cone there corresponds an affine toric variety $X$ and the local ring
$R$ of its unique torus fixed point. The multiplicities of monomial $\m$-primary ideals in $R$ then give covolumes of 
integral polyhedral convex regions in $C$.

More precisely, let $C$ be a closed strictly convex rational polyhedral cone in $\r^n$ with apex at the origin. Consider the additive semigroup 
$\SS = C \cap \z^n$. Consider the semigroup algebra $\k[\SS]$ of $\SS$, that is, the subalgebra of Laurent polynomials $\k[x_1^{\pm 1}, \ldots, x_n^{\pm 1}]$ consisting of all $f(x) = \sum_{\alpha \in \SS} c_\alpha x^\alpha$, where we have used the shorthand notation $x = (x_1, \ldots, x_n)$, $\alpha = (a_1, \ldots, a_n)$ and 
$x^\alpha = x_1^{a_1} \cdots x_n^{a_n}$. The semigroup algebra $\k[\SS]$ is the coordinate ring of the affine troic variety associated to the cone $C$. Finally, let $R$ be the localization of $\k[\SS]$ at the maximal ideal generated by all the nonconstant monomials. 

Let $\Gamma \subset C$ be a $C$-convex region that is also a polyhedron with integral vertices. Consider the set $\I(\Gamma) = \Gamma \cap \z^n$. It is a semigroup ideal in $\SS = C \cap \z^n$, i.e. if $x \in \I(\Gamma)$ and $y \in \SS$ then $x + y \in \I(\Gamma)$. Let $I(\Gamma)$ denote the ideal in $R$ generated by the $x^\alpha$ for all $\alpha \in \I(\Gamma)$. It is easy to see that $I(\Gamma)$ is an $\m$-primary ideal if and only if $\Gamma$ is a cobounded region. 

The following is the local version of the celebrated Kushnirenko theorem (see \cite{Kushnirenko, AVG} for the smooth case and  \cite{KKh-Buch} for the toric case).
\begin{Th}[Local version of Bernstein-Kushnirenko theorem] \label{th-BK-local}
Let $\Gamma_1, \ldots, \Gamma_n$ be cobounded $C$-convex regions that are polyhedra with integral vertices
and let $I(\Gamma_1), \ldots, I(\Gamma_n)$ be their associated monomial ideals.
Then $$e(I(\Gamma_1), \ldots, I(\Gamma_n)) = n!~V(\Gamma_1, \ldots, \Gamma_n).$$
\end{Th}

\begin{proof}[Proof of Theorem \ref{th-covol-poly}]
Let $C^{(j)}$, $j=1, 2, \ldots$  be closed strictly convex rational polyhedral cones in $\r^n$ approximating $C$ arbitrarily closely. Also for each $j$,
let $\Gamma_1^{(j)}, \ldots, \Gamma_n^{(j)}$ be rational polyhedral $C^{(j)}$-convex regions in $\r^n$ approximating respectively $\Gamma_1, \ldots, \Gamma_n$ arbitrarily closely. By Corollary \ref{cor-multi-polynomial} and Theorem \ref{th-BK-local} we know that for each $j$ the function 
$P^{(j)}(\lambda_1, \ldots, \lambda_n) = \covol(\lambda_1 \Gamma^{(j)}_1 + \cdots + \lambda_n \Gamma^{(j)}_n)$ is a homogeneous polynomial of degree $n$ in the $\lambda_i$. But as $j$ goes to infinity the functions $P^{(j)}$ converge to the function $P$ in \eqref{equ-covol-poly}. Since the limit of homogeneous polynomials of degree $n$ is again a homogeneous polynomial of degree $n$ the theorem is proved.
  
\end{proof}

\begin{proof}[Proof of Theorem \ref{th-AF-covolume}]
If $C$ is a rational cone and $\Gamma_1, \ldots, \Gamma_n$ are polyhedral $C$-convex regions with integral vertices, 
the inequality \eqref{equ-AF-covol} follows immediately from Theorem \ref{th-BK-local} and Theorem \ref{th-AF-multi}. From multi-linearity of 
mixed covolume it follows that \eqref{equ-AF-covol} is also true if the $\Gamma_i$ are polyhedral regions with rational vertices.
But any convex region can be approximated arbitrarily closely by polyhedral convex regions with rational vertices. Thus \eqref{equ-AF-covol} 
in general follows by continuity.
\end{proof}

We would like to point out that 
our proof of the (local) Alexandrov-Fenchel inequality for mixed multiplicities relies on the Hodge inequality which is essentially 
the (global) Alexandrov-Fenchel inequality in the intersection theory of subspaces. Similarly, the geometric proof of the Alexandrov-Fenchel 
inequalities for covolumes in \cite{Askold-Vladlen} deduces it from the usual Alexandrov-Fenchel 
inequality for mixed volumes of convex bodies.


\end{document}